\documentclass[12pt,twoside,a4paper]{article}
\usepackage{amsmath,amsthm, amssymb, amsfonts,latexsym,amscd}
\usepackage[dvips]{graphicx}
\newtheorem{theorem}{Theorem}[section]
\newtheorem{corollary}[theorem]{Corollary}
\newtheorem{lemma}[theorem]{Lemma}
\newtheorem{proposition}[theorem]{Proposition}

\newtheorem{example}[theorem]{Example}

\begin{document}

\title{\bf On the minimum cut-sets of the power graph of a finite cyclic group, II}
\author{Sanjay Mukherjee\and Kamal Lochan Patra \and Binod Kumar Sahoo}
\date{}

\maketitle

\begin{abstract}
The power graph $\mathcal{P}(G)$ of a finite group $G$ is the simple graph with vertex set $G$ and two distinct vertices are adjacent if one of them is a power of the other. Let $n=p_1^{n_1}p_2^{n_2}\cdots p_r^{n_r},$ where $p_1,p_2,\ldots,p_r$ are primes with $p_1<p_2<\cdots <p_r$ and $n_1,n_2,\ldots, n_r$ are positive integers. For the cyclic group $C_n$ of order $n$,
the minimum cut-sets of $\mathcal{P}(C_n)$ are characterized in \cite{cps} for $r\leq 3$. Recently, in \cite{MPS}, certain cut-sets of $\mathcal{P}(C_n)$ are identified such that any minimum cut-set of $\mathcal{P}(C_n)$ must be one of them. In this paper, for $r\geq 4$, we explicitly determine the minimum cut-sets, in particular, the vertex connectivity of $\mathcal{P}(C_n)$ when: (i) $n_r\geq 2$, (ii) $r=4$ and $n_r=1$, and (iii) $r=5$, $n_r=1$, $p_1\geq 3$.\\
\end{abstract}

\noindent {\bf Key words:} Power graph, Cut-set, Vertex connectivity, Cyclic group, Euler's totient function\\
{\bf Mathematics subject classification 2020:} 05C25, 05C40, 20K99

\section{\bf Introduction}

Let $\Gamma$ be a finite simple connected graph with vertex set $V$. A subset $X$ of $V$ is called a {\it cut-set} of $\Gamma$ if the induced subgraph of $\Gamma$ with vertex set $V\setminus X$ is  disconnected. Note that $\Gamma$ has a cut-set if and only if $\Gamma$ is not a complete graph. A cut-set $X$ of $\Gamma$ is called a {\it minimum cut-set} if $|X|\leq |Y|$ for any cut-set $Y$ of $\Gamma$. The {\it vertex connectivity} of $\Gamma$, denoted by $\kappa(\Gamma)$, is the minimum number of vertices to be removed from $V$ such that the induced subgraph of $\Gamma$ on the remaining vertices is disconnected or has only one vertex. The latter case arises only when $\Gamma$ is a complete graph. So $\kappa(\Gamma)=|V|-1$ if $\Gamma$ is a complete graph. If $\Gamma$ is not a complete graph and $X$ is a minimum cut-set of $\Gamma$, then $\kappa(\Gamma)=|X|$.

\subsection{Power graphs}

There are various graphs associated with groups which have been studied in the literature, e.g., Cayley graphs, commuting graphs. The notion of directed power graph of a group was introduced by Kelarev and Quinn in \cite{ker1}. The underlying undirected graph, simply termed as the `power graph' of that group, was first considered by Chakrabarty et al.~in \cite{ivy}. Several researchers have then studied both the directed and undirected power graphs of groups from different viewpoints. We refer to the survey papers \cite{AKC, KSCC} and the references therein for more on these graphs.

The {\it power graph} of a finite group $G$, denoted by $\mathcal{P}(G)$, is the simple graph with vertex set $G$ and two distinct vertices are adjacent if one of them is a power of the other. Since the identity element of $G$ is adjacent to all other vertices, $\mathcal{P}(G)$ is a connected graph of diameter at most two. Chakrabarty et al.~proved in  \cite{ivy} that $\mathcal{P}(G)$ is a complete graph if and only if $G$ is a cyclic group of prime power order.

Cameron and Ghosh \cite{cam-2, cam-1} and Mirzargar et al.~\cite{mir, mir-1} have studied the isomorphism problem associated with finite groups and their power graphs. Feng et al.~\cite{FMW-1} described the full automorphism group of the power graph of a finite group. Kirkland et al.~\cite{kirkland} derived explicit formulas concerning the complexity (that is, number of spanning trees) of power graphs of various finite groups. Graph parameters such as chromatic number \cite{MaFeng2015,shitov}, matching number \cite{Cameron-SS}, independence number \cite{ma-fu-lu}, metric dimension \cite{FMW}, strong metric dimension \cite{MaFengWang2018, Ma} etc.~of power graphs of finite groups have also been studied in the literature. It was proved in \cite{dooser, FMW} that the power graph of a finite group is perfect, in particular, the clique number and the chromatic number coincide. Explicit formula for the clique number of the power graph of a finite cyclic group was obtained in \cite{dooser, mir}.

For a given finite group $G$, characterizing the minimum cut-sets of $\mathcal{P}(G)$ in general and determining the vertex connectivity of $\mathcal{P}(G)$ in particular are interesting problems. It was proved in \cite[Theorem 1.3]{cur} and \cite[Corollary 3.4]{cur-1} that, among all finite groups of order $n$, the power graph $\mathcal{P}(C_n)$ of the cyclic group $C_n$ of order $n$ has the maximum number of edges and has the largest clique. It is thus expected that $\kappa(\mathcal{P}(C_n))$ would be larger compared to the vertex connectivity of the power graph of any non-cyclic group of order $n$. In this paper, we explicitly determine the minimum cut-sets and the vertex connectivity of $\mathcal{P}(C_n)$ when $n$ is divisible by the square of its largest prime divisor.

\subsection{Cut-sets of $\mathcal{P}(C_n)$}

Throughout the paper, $n\geq 2$ is an integer and $r$ denotes the number of distinct prime divisors of $n$. We write the prime power factorization of $n$ as
$$n=p_1^{n_1}p_2^{n_2}\cdots p_r^{n_r},$$
where $p_1,p_2,\ldots,p_r$ are primes with $p_1<p_2<\cdots <p_r$ and $n_1,n_2,\ldots, n_r$ are positive integers. For a positive integer $m$, we denote the set $\{1,2,\dots,m\}$ by $[m]$.

For every positive divisor $d$ of $n$, there is a unique (cyclic) subgroup of $C_n$ of order $d$. It follows that two distinct elements $x,y$ of $C_n$ are adjacent in $\mathcal{P}(C_n)$ if and only if $o(x)\mid o(y)$ or $o(y)\mid o(x)$, where $o(z)$ denotes the order of an element $z\in C_n$. The lattice of all subgroups of $C_n$ with respect to inclusion is isomorphic to the lattice of all divisors of $n$ with respect to divisibility. If $H_1,H_2,\ldots ,H_k$ are subgroups of $C_n$, then the number of elements in the intersection $H_1\cap H_2\cap\ldots\cap H_k$ is equal to the greatest common divisor of the integers $|H_1|,|H_2|,\ldots ,|H_k|$.

For a positive divisor $d$ of $n$, we denote by $E_d$ the set of all elements of $C_n$ whose order is $d$, and by $S_d$ the set of all elements of $C_n$ whose order divides $d$.
Then $S_d$ is the unique (cyclic) subgroup of $C_n$ of order $d$ and $E_d$ is precisely the set of generators of $S_d$. Thus $\left\vert S_d\right\vert =d$ and $\left\vert E_d\right\vert =\phi(d)$, where $\phi$ is the Euler's totient function. Any two distinct elements of $E_d$ are adjacent in $\mathcal{P}(C_n)$.

If $r=1$, then $n$ is a prime power and so $\mathcal{P}(C_n)$ is a complete graph. In this case, there is no cut-set of $\mathcal{P}(C_n)$ and $\kappa\left(\mathcal{P}\left(C_n\right)\right)=n-1$. Therefore, we shall consider that $r\geq 2$. Any cut-set of $\mathcal{P}(C_n)$ must contain the sets $E_n$ and $E_1$, as each element of $E_n\cup E_1$ is adjacent with all other vertices of $\mathcal{P}(C_n)$. By \cite[Proposition 3.10]{cps-2}, any minimum cut-set of $\mathcal{P}(C_n)$ contains at most two of the sets $E_{\frac{n}{p_1}}, E_{\frac{n}{p_2}},\ldots, E_{\frac{n}{p_r}}$. If $n_r\geq 2$, then any minimum cut-set of $\mathcal{P}(C_n)$ contains at most one of $E_{\frac{n}{p_1}}, E_{\frac{n}{p_2}},\ldots, E_{\frac{n}{p_r}}$ by \cite[Proposition 4.6]{cps-2}.

In \cite{MPS}, the authors identified certain cut-sets of $\mathcal{P}(C_n)$ such that any minimum cut-set of $\mathcal{P}(C_n)$ must be one of them. These cut-sets are of two types as described below.\bigskip

\noindent\underline{\bf Cut-sets of the first type}:\medskip\\
Let $r\geq 2$, $a\in [r]$ and $s\in [n_a]$. Denote by $Q_a^s$ the union of the subgroups $S_{\frac{n}{p_ip_a^s}}$ of $C_n$, where $i\in [r]\setminus\{a\}$. Define the subset $Z_a^s$ of $C_n$ by
$$Z_a^s:=E_n\cup E_{\frac{n}{p_a}}\cup E_{\frac{n}{p_a^2}}\cup\ldots\cup E_{\frac{n}{p_a^{s-1}}}\cup Q_a^s$$
and the integer $\beta_a^s$ by
$$\beta_a^s:=\phi(n) + \frac{n}{p_1p_2\cdots p_r}\cdot \frac{1}{p_a^{s-1}}\cdot \left[\frac{p_1p_2\cdots p_{r}}{p_a} +\phi\left(\frac{p_1p_2\cdots p_{r}}{p_a}\right)\left(p_a^{s-1}-2\right) \right].$$
Note that $Z_a^s$ contains none or exactly one of the sets $E_{\frac{n}{p_1}}, E_{\frac{n}{p_2}},\ldots, E_{\frac{n}{p_r}}$ according as $s=1$ or $s>1$ (in the latter case, $Z_a^s$ contains $E_{\frac{n}{p_a}}$).
By \cite[Propositions 3.1, 3.3]{MPS}, $Z_a^s$ is a cut-set of $\mathcal{P}(C_n)$ of size $\beta_a^s$ for $r\geq 2$. We thus have the following:

\begin{proposition}[\cite{MPS}]\label{prop-Z-a-s}
If $r\geq 2$, then $\kappa(\mathcal{P}(C_n))\leq \beta_a^s=|Z_a^s|$ for every $a\in [r]$ and $s\in [n_a]$.\smallskip
\end{proposition}

\noindent\underline{\bf Cut-sets of the second type}:\medskip\\
Let $r\geq 3$, $a,b\in [r]$ with $a\neq b$, $s\in [n_a]$ and $t\in [n_b]$. Denote by $K_{a,b}^{s,t}$ the set of all the non-generators of the cyclic subgroup $S_{\frac{n}{p_a^s p_b^t}}$ of $C_n$, and by $H_{a,b}^{s,t}$ the union of the mutually disjoint sets $E_{\frac{n}{p_a^ip_b^j}}$, where $0\leq i\leq s$, $0\leq j\leq t$ and $(i,j)\neq (s,t)$. Define the subset $X_{a,b}^{s,t}$ of $C_n$ by
$$X_{a,b}^{s,t}:=H_{a,b}^{s,t}\cup K_{a,b}^{s,t}.$$
Note that $X_{a,b}^{s,t}=X_{b,a}^{t,s}$ and it contains exactly two of the sets $E_{\frac{n}{p_1}}, E_{\frac{n}{p_2}},\ldots, E_{\frac{n}{p_r}}$, namely $E_{\frac{n}{p_a}}$ and $E_{\frac{n}{p_b}}$. By \cite[Proposition 3.5]{MPS}, $X_{a,b}^{s,t}$ is a cut-set of $\mathcal{P}(C_n)$ for $r\geq 3$. We shall determine the cardinality of $X_{a,b}^{s,t}$ in Section $4$.

\subsection{Known results}\label{known-results}

Suppose that $r=2$. Then $n=p_1^{n_1}p_2^{n_2}$. In this case, $\kappa\left(\mathcal{P}\left(C_n\right)\right)=\beta_2^1$ by \cite[Theorem 2.38]{panda}, also see \cite[Theorem 1.3(i),(iii)]{cps}. Further, if $p_1 \geq 3$, then $2\phi(p_1) >p_1$ and so $Z_2^1$ is the only minimum cut-set of $\mathcal{P}(C_n)$ by \cite[Proposition 4.4]{cps}.
If $p_1=2$, then it follows from the proof of \cite[Theorem 1.3(iii)]{cps} that the number of distinct minimum cut-sets of $\mathcal{P}(C_n)$ is $n_2$, namely the cut-sets $Z_2^s$ for $s\in [n_2]$.

Suppose that $r=3$. Then $n=p_1^{n_1}p_2^{n_2}p_3^{n_3}$. If $p_1\geq 3$, then $2\phi(p_1p_2) >p_1p_2$ and so $\kappa\left(\mathcal{P}\left(C_n\right)\right)=\beta_3^1$ by \cite[Theorem 1.3(i)]{cps}. In this case, $Z_3^1$ is the only minimum cut-set of $\mathcal{P}(C_n)$ by \cite[Proposition 4.4]{cps}. If $p_1=2$, then $\kappa\left(\mathcal{P}\left(C_n\right)\right)=\beta_3^{n_3} $ by \cite[Theorem 1.5]{cps}. In this case, it follows from the proof of \cite[Theorem 1.5]{cps} that $Z_3^{n_3}$ is the only minimum cut-set of $\mathcal{P}(C_n)$.

Suppose that $r\geq 4$. Then $2\phi(p_1p_2\cdots p_{r-1})\neq p_1p_2\cdots p_{r-1}$. If $2\phi(p_1p_2\cdots p_{r-1})> p_1p_2\cdots p_{r-1}$, then $\kappa\left(\mathcal{P}\left(C_n\right)\right)=\beta_r^1$ by \cite[Theorem 1.3(i)]{cps} and $Z_r^1$ is the only minimum cut-set of $\mathcal{P}(C_n)$ by \cite[Proposition 4.4]{cps}. Therefore, we may assume that $2\phi(p_1p_2\cdots p_{r-1}) < p_1p_2\cdots p_{r-1}$.
If $n=p_1p_2\cdots p_r$, then $\kappa\left(\mathcal{P}\left(C_n\right)\right)=\min\left\{\beta_r^1,|X_{r-1,r}^{1,1}|\right\}$ by \cite[Theorem 1.3]{cps-2} (see Proposition \ref{thm-X-a-b-s-t} for $|X_{r-1,r}^{1,1}|$). In this case, it follows from the proof of \cite[Theorem 1.3]{cps-2} that if $Y$ is a minimum cut-set of $\mathcal{P}(C_n)$, then $Y\in \{Z_r^1, X_{r-1,r}^{1,1}\}$. Further, by \cite[Proposition 3.9(ii)]{cps-2}, $\kappa\left(\mathcal{P}\left(C_n\right)\right)=|X_{r-1,r}^{1,1}|$ if and only if $\left(2+\frac{p_r-2}{p_{r-1}-1}\right)\cdot\phi(p_1p_2\cdots p_{r-2})\leq  p_1p_2\cdots p_{r-2}$.

For general $n$ with $r\geq 4$, it was proved in \cite[Theorem 1.1]{MPS} that if $X$ is a minimum cut-set of $\mathcal{P}(C_n)$, then $X=Z_r^1$; or $X=Z_a^{n_a}$ for some $a\in [r]$ with $n_a\geq 2$; or $X=X_{a,b}^{s,t}$ for some $a,b\in [r]$ with $a\neq b$, $s\in [n_a]$ and $t\in [n_b]$.
Indeed, when $n_r\geq 2$, we have $X=Z_r^1$ or $X=Z_a^{n_a}$ for some $a\in [r]$ with $n_a\geq 2$, giving $\kappa(\mathcal{P}(C_n))=\min\{\beta_r^{1},\beta_a^{n_a}: a\in [r], n_a\geq 2\}$ which was obtained in \cite[Theorem 1.2]{cps-2-E}).

\subsection{Main results}

In view of the results mentioned above, the problems of characterizing the minimum cut-sets and determining the vertex connectivity of $\mathcal{P}(C_n)$ are still open when $r\geq 4$, $2\phi(p_1p_2\cdots p_{r-1}) < p_1p_2\cdots p_{r-1}$ and $n_i >1$ for at least one $i\in [r]$.

When $n_r\geq 2$, we prove the following theorem in Section \ref{sec-main-1} which contributes further to the study of the minimum cut-sets and the vertex connectivity of $\mathcal{P}(C_n)$.

\begin{theorem}\label{main-1}
Let $r\geq 4$, $n_r\geq 2$, $2\phi(p_1p_2\cdots p_{r-1}) < p_1p_2\cdots p_{r-1}$ and $X$ be a minimum cut-set of $\mathcal{P}(C_n)$. Consider the subset $\Omega$ of $[r-1]$ defined by
$$\Omega:=\left\{j\in [r-1]\mid n_j\geq 3,\; p_r-p_j \leq r-4,\; 2\phi\left(\frac{p_1p_2\cdots p_{r}}{p_j}\right) < \frac{p_1p_2\cdots p_{r}}{p_j}\right\}.$$
Then the following hold:
\begin{enumerate}
\item[(i)] If $n_r\geq 3$ or $\Omega=\emptyset$, then $X=Z_r^{n_r}$ and so $\kappa\left(\mathcal{P}\left(C_n\right)\right)=\beta_r^{n_r}$;
\item[(ii)] If $n_r=2$ and $\Omega\neq\emptyset$, then $X=Z_r^{n_r}$ or $X=Z_b^{n_b}$, where $b$ is the largest integer in $\Omega$. As a consequence, we have $\kappa\left(\mathcal{P}\left(C_n\right)\right)\in \{\beta_b^{n_b},\beta_r^{n_r}\}$.
\end{enumerate}
\end{theorem}

Note that the set $\Omega$ defined in Theorem \ref{main-1} is the empty set for $r\in\{4,5\}$. Therefore, by Theorem \ref{main-1}(i), we have the following corollary.

\begin{corollary}
Let $r\in\{4,5\}$, $n_r\geq 2$, $2\phi(p_1p_2\cdots p_{r-1}) < p_1p_2\cdots p_{r-1}$ and $X$ be a minimum cut-set of $\mathcal{P}(C_n)$. Then $X=Z_r^{n_r}$ and so $\kappa(\mathcal{P}(C_n))=\beta_r^{n_r}$.
\end{corollary}

In Example \ref{example}, we provide integers $n$ with $r=6$, $n_r=2$ and $\Omega\neq\emptyset$ such that each of the possibilities for $\kappa(\mathcal{P}(C_n))$ and so for $X$ in Theorem \ref{main-1}(ii) occurs. For $r\in\{4,5\}$ with $n_r=1$, we prove the following theorem in Section \ref{sec-main-2}.

\begin{theorem}\label{main-2}
Let $r\in\{4,5\}$, $n_r=1$, $2\phi(p_1p_2\cdots p_{r-1})< p_1p_2\cdots p_{r-1}$ and $X$ be a minimum cut-set of $\mathcal{P}(C_n)$. Then $p_1\in \{2,3\}$ and the following hold:
\begin{enumerate}
\item[(i)] If $r=4$, then $X\in \{Z_3^{n_3}, Z_4^{1}\}$ and so $\kappa\left(\mathcal{P}\left(C_n\right)\right)\in \{\beta_3^{n_3},\beta_4^{1}\}$. Furthermore, if $n_{3}=1$, then $X= Z_4^{1}$ and so $\kappa\left(\mathcal{P}\left(C_n\right)\right)=\beta_4^{1}$.
\item[(ii)] If $r=5$ and $p_1 = 3$, then $X\in \{Z_4^{n_4}, Z_5^{1}\}$ and so $\kappa\left(\mathcal{P}\left(C_n\right)\right)\in \{\beta_4^{n_4},\beta_5^{1}\}$. Furthermore, if $n_{4}=1$, then $X= Z_5^{1}$ and so $\kappa\left(\mathcal{P}\left(C_n\right)\right)=\beta_5^{1}$.
\end{enumerate}
\end{theorem}

We note that when $r=5$ with $p_1=2$, $n_{r-1}=1=n_r$ and $2\phi(p_1p_2\cdots p_{r-1})< p_1p_2\cdots p_{r-1}$, it is not necessary that $Z_5^{1}$ is a minimum cut-set of $\mathcal{P}(C_n)$ (see Example \ref{example-0}). In Examples \ref{example-1} and \ref{example-2}, we provide integers $n$ with $r\in\{4,5\}$, $n_r=1$, $n_{r-1}\geq 2$ and $2\phi(p_1p_2\cdots p_{r-1})< p_1p_2\cdots p_{r-1}$ such that $\kappa(\mathcal{P}(C_n))=\beta_{r-1}^{n_{r-1}}$ and so $X=Z_{r-1}^{n_{r-1}}$ in (i) and (ii) of Theorem \ref{main-2}.

\section{\bf Preliminaries}

Recall that $\phi$ is a multiplicative function, that is, $\phi(m_1m_2)=\phi(m_1)\phi(m_2)$ for any two positive integers $m_1,m_2$ that are relatively prime. Therefore, $\phi(n)= \phi\left(p_1^{n_1}\right)\phi\left(p_2^{n_2}\right) \cdots \phi\left(p_r^{n_r}\right)$, where $\phi\left(p_i^{n_i}\right)=p_i^{n_i-1}(p_i-1)=p_i^{n_i-1}\phi(p_i)$ for $i\in [r]$. We also know that $\underset{d\mid m}{\sum} \phi(d) = m$ for every positive integer $m$.

In this section, we consider $r\geq 2$. For $j\in [r]$, we have $p_j^{n_j -1} \cdot\phi\left(\frac{n}{p_j^{n_j}}\right)\geq \phi\left(\frac{n}{p_j}\right)$, with equality if and only if $n_j=1$. Let $i,k\in [r]$ with $i<k$. Applying a similar argument as in the proof of \cite[Lemma 2.1]{cps}, we get
\begin{equation}\label{eqn-0}
\phi\left(\frac{n}{p_i}\right)\geq p_k^{n_k -1}\cdot \phi\left(\frac{n}{p_k^{n_k}}\right),
\end{equation}
where the inequality (\ref{eqn-0}) is strict except when $(i,k)=(1,2)$, $(p_1,p_2)=(2,3)$ and $n_1\geq 2$.
As a consequence, we have
\begin{equation}\label{eqn-1}
\left|E_{\frac{n}{p_i}} \right|=\phi\left(\frac{n}{p_i}\right)\geq \phi\left(\frac{n}{p_k}\right)=\left|E_{\frac{n}{p_k}} \right|,
\end{equation}
where equality holds in (\ref{eqn-1}) if and only if $(i,k)=(1,2)$, $(p_1, p_2) =(2,3)$, $n_1\geq 2$ and $n_2 = 1$.
If $p_1\geq r$, then
\begin{equation}\label{eqn-1-1}
2 \phi(p_1p_2\cdots p_{r-1}) \geq p_1p_2\cdots p_{r-1},
\end{equation}
with equality if and only if $(r,p_1)= (2,2)$. This can be seen from the following:
\begin{equation*}
\begin{aligned}
\frac{\phi(p_1p_2\cdots p_{r-1})}{p_1p_2\cdots p_{r-1}} & = \left(1-\frac{1}{p_1}\right)\left(1-\frac{1}{p_2}\right)\cdots \left(1-\frac{1}{p_{r-1}}\right)\\
& \geq \left(1-\frac{1}{r}\right)\left(1-\frac{1}{r+1}\right)\cdots \left(1-\frac{1}{2(r-1)}\right)= \frac{r-1}{2(r-1)}=\frac{1}{2},
\end{aligned}
\end{equation*}
where the second inequality is strict except when $r=2$ and $p_1=2$.
Let $I$ be a non-empty subset of $[r]$ with $|I|=t$. By \cite[Lemma 3.1]{cps-1}, we have
\begin{equation}\label{eqn-2}
(t+1)\phi\left(\underset{i\in I}\prod p_i\right) \geq \underset{i\in I}\prod p_i,
\end{equation}
where equality holds if and only if either $(t,I)=(1,\{1\})$ with $p_1=2$ or $(t,I)=(2,\{1,2\})$ with $p_1=2$ and $p_2=3$.
In fact, if $k\cdot\phi\left(\underset{i\in I}\prod p_i\right) = \underset{i\in I}\prod p_i$ for some positive integer $k$, then either $(t,I, k)=(1,\{1\},2)$ with $p_1=2$, or $(t,I,k)=(2,\{1,2\},3)$ with $p_1=2$ and $p_2=3$.

%

\begin{lemma}\label{comparison-1}
Let $a,b\in [r]$ with $a< b$. If $\mu \phi\left(\frac{p_1p_2\cdots p_{r}}{p_a}\right) < \frac{p_1p_2\cdots p_{r}}{p_a}$ for some positive real number $\mu$, then $\mu\phi\left(\frac{p_1p_2\cdots p_{r}}{p_b}\right) < \frac{p_1p_2\cdots p_{r}}{p_b}$.
\end{lemma}

\begin{proof}
Since $a< b$, we have $p_a <p_b$ and so $\frac{\phi(p_a)}{\phi(p_b)} < \frac{p_a}{p_b}$. Using $\mu \phi\left(\frac{p_1p_2\cdots p_{r}}{p_a}\right) < \frac{p_1p_2\cdots p_{r}}{p_a}$, we get
$\mu\phi\left(\frac{p_1p_2\cdots p_{r}}{p_b}\right) =\mu\phi\left(\frac{p_1p_2\cdots p_{r}}{p_a}\right)\cdot \frac{\phi(p_a)}{\phi(p_b)} < \frac{p_1p_2\cdots p_{r}}{p_a}\cdot \frac{p_a}{p_b}=\frac{p_1p_2\cdots p_{r}}{p_b}$.
\end{proof}

\begin{lemma}\label{comparison-2}
Let $a,b\in [r]$ with $a < b$. If $\mu\phi\left(\frac{p_1p_2\cdots p_{r}}{p_b}\right) > \frac{p_1p_2\cdots p_{r}}{p_b}$ for some positive real number $\mu$, then $\mu\phi\left(\frac{p_1p_2\cdots p_{r}}{p_a}\right) > \frac{p_1p_2\cdots p_{r}}{p_a}$.
\end{lemma}

\begin{proof}
Since $a < b$, we have $p_a < p_b$ and so $\frac{\phi(p_b)}{\phi(p_a)} > \frac{p_b}{p_a}$. Using $\mu\phi\left(\frac{p_1p_2\cdots p_{r}}{p_b}\right) > \frac{p_1p_2\cdots p_{r}}{p_b}$, we get
$\mu\phi\left(\frac{p_1p_2\cdots p_{r}}{p_a}\right) =\mu\phi\left(\frac{p_1p_2\cdots p_{r}}{p_b}\right)\cdot \frac{\phi(p_b)}{\phi(p_a)} > \frac{p_1p_2\cdots p_{r}}{p_b}\cdot \frac{p_b}{p_a}=\frac{p_1p_2\cdots p_{r}}{p_a}$.
\end{proof}

The following result is in the spirit of the inequality (\ref{eqn-1-1}) and an improvement of (\ref{eqn-2}) under additional condition on the smallest prime $p_1$.

\begin{lemma}\label{p1_lower_bound}
If $p_1\geq \frac{r+2}{2}$, then $3\phi(p_1p_2\cdots p_r)\geq p_1p_2\cdots p_r$ and equality holds if and only if $(r,p_1,p_2)=(2,2,3)$.
\end{lemma}

\begin{proof}
Since $p_1\geq \frac{r+2}{2}$, we get $p_k\geq \frac{r+2}{2} + k-1 = \frac{r+2k}{2}$ for every $k\in [r]$ and the inequality is strict for $k\geq 3$. Then
\begin{align*}
\frac{\phi(p_1p_2\cdots p_r)}{p_1p_2\cdots p_r} & =\left(1-\frac{1}{p_1}\right)\left(1-\frac{1}{p_2}\right)\cdots \left(1-\frac{1}{p_r}\right)\\
 & \geq\left(1-\frac{2}{r+2}\right)\left(1-\frac{2}{r+4}\right)\cdots \left(1-\frac{2}{3r}\right)=\frac{r}{3r}=\frac{1}{3}.
\end{align*}
The second inequality in the above is strict when $r\geq 3$. For $r=2$, it follows from the equality case of (\ref{eqn-2}) that $3\phi(p_1p_2)=p_1p_2$ if and only if $(p_1,p_2)=(2,3)$.
\end{proof}

The following result was proved in \cite[Lemma 2.1]{MPS}.

\begin{lemma}[\cite{MPS}]\label{lem-2.1}
Let $a\in [r]$ and $0\leq k\leq s\leq n_a$. Then
$$
\underset{l=k}{\overset{s}{\sum}} \left\vert E_{\frac{n}{p_a^{l}}} \right\vert =\underset{l=k}{\overset{s}{\sum}} \phi\left( \frac{n}{p_a^{l}}\right)=
\begin{cases}
\frac{n}{p_1p_2\cdots p_r}\cdot \frac{1}{p_a^{k-1}}\cdot \phi\left(\frac{p_1p_2\cdots p_r}{p_a} \right), &\text{if $s=n_a$}\\
\frac{n}{p_1p_2\cdots p_r}\cdot \frac{1}{p_a^{k-1}}\cdot\phi\left(\frac{p_1p_2\cdots p_r}{p_a} \right)\left(1-\frac{1}{p_a^{s-k+1}}\right), &\text{if $s<n_a$}.
\end{cases}
$$
\end{lemma}

\begin{lemma}\label{lem-2.3}
Let $a,b\in [r]$ with $a\neq b$, $s\in [n_a]$ and $t\in [n_b]$. Then
$$
\underset{k=1}{\overset{s}{\sum}} \; \underset{l=1}{\overset{t}{\sum}} \left\vert E_{\frac{n}{p_a^{k}p_b^{l}}} \right\vert
=\left\{\begin{array}{ll}
    \frac{n}{p_1p_2\cdots p_r}\cdot\phi\left(\frac{p_1p_2\cdots p_r}{p_ap_b}\right), & \text{if } s=n_a,\; t=n_b \\
    \frac{n}{p_1p_2\cdots p_r}\cdot\phi\left(\frac{p_1p_2\cdots p_r}{p_ap_b}\right)\left(1-\frac{1}{p_b^{t}}\right), & \text{if } s=n_a,\; t<n_b \\
    \frac{n}{p_1p_2\cdots p_r}\cdot\phi\left(\frac{p_1p_2\cdots p_r}{p_ap_b}\right)\left(1-\frac{1}{p_a^{s}}\right), & \text{if } s<n_a,\; t=n_b \\
    \frac{n}{p_1p_2\cdots p_r}\cdot\phi\left(\frac{p_1p_2\cdots p_r}{p_ap_b}\right)\left(1-\frac{1}{p_a^{s}}\right)\left(1-\frac{1}{p_b^{t}}\right), & \text{if } s<n_a,\; t<n_b.
\end{array}
\right.
$$
\end{lemma}

\begin{proof}
We have $\underset{k=1}{\overset{s}{\sum}}\; \underset{l=1}{\overset{t}{\sum}} \left\vert E_{\frac{n}{p_a^{k}p_b^{l}}} \right\vert = \underset{k=1}{\overset{s}{\sum}}\; \underset{l=1}{\overset{t}{\sum}} \phi\left( \frac{n}{p_a^{k}p_b^{l}}\right) =\phi\left(\frac{n}{p_a^{n_a}p_b^{n_b}}\right) \left[\underset{k=1}{\overset{s}{\sum}}\phi\left( p_a^{n_a -k}\right)\right] \left[\underset{l=1}{\overset{t}{\sum}} \phi\left( p_b^{n_b - l}\right) \right]$.
Using the three equalities that
$$\phi\left( \frac{n}{p_a^{n_a}p_b^{n_b}}\right)= \frac{n}{p_1p_2\cdots p_r}\cdot \frac{1}{p_a^{n_a-1}p_b^{n_b-1}}\cdot \phi\left(\frac{p_1p_2\cdots p_r}{p_ap_b} \right),$$
$$\underset{k=1}{\overset{s}{\sum}} \phi\left(p_a^{n_a-k}\right)=\underset{i=n_a-s}{\overset{n_a-1}{\sum}} \phi(p_a^i)= \left\{
\begin{array}{ll}
    p_a^{n_a-1}, & \text{if } s=n_a\\
    p_a^{n_a-1}\left(1-\frac{1}{p_a^{s}}\right), & \text{if } s<n_a\\
\end{array}
\right.$$
and
$$\underset{l=1}{\overset{t}{\sum}} \phi\left(p_b^{n_b-l}\right)=\underset{j=n_b-t}{\overset{n_b-1}{\sum}} \phi(p_b^j)= \left\{
\begin{array}{ll}
    p_b^{n_b-1}, & \text{if } t=n_b\\
    p_b^{n_b-1}\left(1-\frac{1}{p_b^{t}}\right), & \text{if } t<n_b,\\
\end{array}
\right.$$
the lemma now follows.
\end{proof}

\section{\bf The upper bounds $\beta_a^s$ for $\kappa(\mathcal{P}(C_n))$}

In this section, we consider $r\geq 3$. For $a\in [r]$ and $s\in [n_a]$, recall that $\beta_a^s$ is an upper bound for $\kappa(\mathcal{P}(C_n))$, where
\begin{equation}\label{eqn-5}
\beta_a^s=\phi(n) + \frac{n}{p_1p_2\cdots p_r}\cdot \frac{1}{p_a^{s-1}}\cdot \left[\frac{p_1p_2\cdots p_{r}}{p_a} +\phi\left(\frac{p_1p_2\cdots p_{r}}{p_a}\right)\left(p_a^{s-1}-2\right) \right].
\end{equation}
The next result follows from \cite[Proposition 3.4]{MPS}, which compares the bounds $\beta_a^s$ and $\beta_a^t$ for a given $a\in [r]$ and distinct $s,t\in [n_a]$.

\begin{lemma}[\cite{MPS}]\label{comp-Z-a-s}
The following hold for $a\in [r]$ and $n_a\geq 2:$
\begin{enumerate}
\item[(i)] If $2\phi\left(\frac{p_1p_2\cdots p_r}{p_a}\right)>\frac{p_1p_2\cdots p_r}{p_a}$, then $\beta_a^1 < \beta_a^2 <\cdots <  \beta_a^{n_a} $.
\item[(ii)] If $2\phi\left(\frac{p_1p_2\cdots p_r}{p_a}\right)<\frac{p_1p_2\cdots p_r}{p_a}$, then $\beta_a^1 > \beta_a^2 >\cdots >  \beta_a^{n_a} $.
\end{enumerate}
\end{lemma}

\begin{proposition}\label{compa}
Let $a,b\in [r]$ with $a<b$. If $p_a^{s}<p_b$ for some $s\in [n_a]$, then $\beta_a^{s}>\beta_b^1$.
\end{proposition}

\begin{proof}
We have $\beta_a^{s}=\phi(n)+\frac{n}{p_1p_2\cdots p_r}\cdot u_a$ and $\beta_b^{1}=\phi(n)+\frac{n}{p_1p_2\cdots p_r}\cdot u_b$, where
$$u_a=\frac{1}{p_a^{s -1}}\cdot \left[\frac{p_1p_2\cdots p_r}{p_a} + \left(p_a^{s-1} -2\right)\phi\left(\frac{p_1p_2\cdots p_r}{p_a}\right)\right]$$
and
$$u_b=\frac{p_1p_2\cdots p_r}{p_b} - \phi\left(\frac{p_1p_2\cdots p_r}{p_b}\right).$$
It can be calculated that
$$u_a -u_b = \frac{1}{p_a^{s -1}}\cdot \left[\left[\left(p_a^{s-1}-2\right)\phi(p_b)+p_a^{s-1}\phi(p_a)\right]\phi\left(\frac{p_1p_2\cdots p_r}{p_a p_b}\right) + \left(p_b-p_a^{s}\right) \frac{p_1p_2\cdots p_r}{p_a p_b} \right].$$
If $s\geq 2$, then clearly $u_a -u_b >0$ as $p_a^{s}<p_b$. If $s=1$, then
$$u_a -u_b = \left(p_b-p_a\right)\left[\frac{p_1p_2\cdots p_r}{p_a p_b} - \phi\left(\frac{p_1p_2\cdots p_r}{p_a p_b}\right)\right]>0$$
as $p_a<p_b$ and $r\geq 3$. It follows that $\beta_a^{s}>\beta_b^{1}$.
\end{proof}

Taking $s=1$ in Proposition \ref{compa}, we get the following:

\begin{corollary}\label{cor-compa}
$\beta_1^1 >\beta_2^1>\cdots >\beta_r^1$.
\end{corollary}

\begin{proposition}\label{compa-1}
We have $\beta_r^1 < \beta_1^{n_1}$. As a consequence, $\kappa(\mathcal{P}(C_n))< \beta_1^{n_1}$ and so $Z_1^{n_1}$ can never be a minimum cut-set of $\mathcal{P}(C_n)$.
\end{proposition}

\begin{proof}
If $\beta_1^1 \leq \beta_1^{n_1}$, then the result follows applying Corollary \ref{cor-compa}. Assume that $\beta_1^1 > \beta_1^{n_1}$. Since $r\geq 3$, we have $2\phi\left(p_2p_3\cdots p_{r}\right)\neq p_2p_3\cdots p_{r}$. Therefore,
$2\phi\left(p_2p_3\cdots p_{r}\right) < p_2p_3\cdots p_{r}$ by Lemma \ref{comp-Z-a-s}. Then from (\ref{eqn-5}), we get
\begin{align*}
\beta_1^{n_1}&=\phi(n) + \frac{n}{p_1p_2\cdots p_r}\cdot \left[\phi\left(p_2p_3\cdots p_{r}\right) + \frac{1}{p_1^{n_1-1}}\cdot \left[ p_2p_3\cdots p_{r} -2\phi\left(p_2p_3\cdots p_{r}\right)\right] \right]\\
 & > \phi(n) + \frac{n}{p_1p_2\cdots p_r}\cdot \phi\left(p_2p_3\cdots p_{r}\right) > \phi(n) + \frac{n}{p_1p_2\cdots p_r}\cdot p_1p_2\cdots p_{r-1}.
\end{align*}
This gives $\beta_1^{n_1}> \phi(n) + \frac{n}{p_1p_2\cdots p_r}\cdot \left[p_1p_2\cdots p_{r-1}-\phi(p_1p_2\cdots p_{r-1}) \right] = \beta_r^1$. The remaining part follows from Proposition \ref{prop-Z-a-s}.
\end{proof}

The following result is some kind of a generalization of Proposition \ref{compa} under certain conditions.

\begin{lemma}\label{compa-2}
Let $a,b\in [r]$ with $a< b$, $s\in [n_a]$ and $t\in [n_b]$. If $p_a^{s-1}\leq p_b^{t-1}$ and $2\phi\left(\frac{p_1p_2\cdots p_{r}}{p_a}\right) < \frac{p_1p_2\cdots p_{r}}{p_a}$, then $\beta_a^{s}>\beta_b^{t}$.
\end{lemma}

\begin{proof}
From (\ref{eqn-5}), we have $\beta_a^{s}=\phi(n)+\frac{n}{p_1p_2\cdots p_r}\cdot u_a$ and $\beta_b^{t}=\phi(n)+\frac{n}{p_1p_2\cdots p_r}\cdot u_b$, where
$$u_a=\frac{1}{p_a^{s -1}}\cdot \left[\frac{p_1p_2\cdots p_r}{p_a} - 2\phi\left(\frac{p_1p_2\cdots p_r}{p_a}\right)\right] + \phi\left(\frac{p_1p_2\cdots p_r}{p_a}\right)$$
and
$$u_b=\frac{1}{p_b^{t -1}}\cdot \left[\frac{p_1p_2\cdots p_r}{p_b} - 2\phi\left(\frac{p_1p_2\cdots p_r}{p_b}\right)\right] + \phi\left(\frac{p_1p_2\cdots p_r}{p_b}\right).$$
Since $p_a^{s-1}\leq p_b^{t-1}$, we have $\frac{1}{p_a^{s -1}} \geq \frac{1}{p_b^{t -1}}$. Then $2\phi\left(\frac{p_1p_2\cdots p_{r}}{p_a}\right) < \frac{p_1p_2\cdots p_{r}}{p_a}$ implies that
$$u_a \geq \frac{1}{p_b^{t-1}}\cdot \left[\frac{p_1p_2\cdots p_r}{p_a} - 2\phi\left(\frac{p_1p_2\cdots p_r}{p_a}\right)\right] + \phi\left(\frac{p_1p_2\cdots p_r}{p_a}\right).$$
It can be calculated that
$$u_a -u_b \geq \frac{p_b -p_a}{p_b^{t -1}}\cdot \left[\frac{p_1p_2\cdots p_r}{p_a p_b} +\left(p_b^{t-1}-2\right)\phi\left(\frac{p_1p_2\cdots p_r}{p_a p_b}\right) \right] > 0,$$
where the last inequality follows using the facts that $r\geq 3$ and $a< b$. Thus $u_a >u_b$ and hence $\beta_a^{s}>\beta_b^{t}$.
\end{proof}

In the following proposition, for $a,b\in [r]$ with $a< b$ and $n_b\geq 2$, we compare the bounds $\beta_a^{n_a}$ and $\beta_b^{n_b}$ under certain conditions.

\begin{proposition}\label{compa-3}
Let $a,b\in [r]$ such that $a< b$ and $2\phi\left(\frac{p_1p_2\cdots p_{r}}{p_a}\right) < \frac{p_1p_2\cdots p_{r}}{p_a}$. Then the following hold:
\begin{enumerate}
\item[(i)] If $n_b\geq 3$ and $p_{b}^2 (p_b -p_a)+2\phi(p_a)\geq (r-1)p_a$, then $\beta_a^{n_a} > \beta_b^{n_b}$.
\item[(ii)] If $n_b=2$ and $p_b(p_b -p_{a})+2\phi(p_{a}) \geq (r-1)p_a$, then $\beta_a^{n_a} > \beta_b^{n_b}$.
\end{enumerate}
\end{proposition}

\begin{proof}
Since $a <b$ and $2\phi\left(\frac{p_1p_2\cdots p_{r}}{p_a}\right) < \frac{p_1p_2\cdots p_{r}}{p_a}$, we have $2\phi\left(\frac{p_1p_2\cdots p_{r}}{p_b}\right) < \frac{p_1p_2\cdots p_{r}}{p_b}$ by Lemma \ref{comparison-1}. We can write $\beta_a^{n_a}=\phi(n)+\frac{n}{p_1p_2\cdots p_r}\cdot u_a$ and $\beta_b^{n_b}=\phi(n)+\frac{n}{p_1p_2\cdots p_r}\cdot u_b$, where
\begin{align*}
u_a & =\frac{1}{p_a^{n_a -1}}\cdot \left[\frac{p_1p_2\cdots p_r}{p_a} - 2\phi\left(\frac{p_1p_2\cdots p_r}{p_a}\right)\right] + \phi\left(\frac{p_1p_2\cdots p_r}{p_a}\right) > \phi\left(\frac{p_1p_2\cdots p_r}{p_a}\right)
\end{align*}
and
\begin{align*}
u_b & =\frac{1}{p_b^{n_b -1}}\cdot \left[\frac{p_1p_2\cdots p_r}{p_b} - 2\phi\left(\frac{p_1p_2\cdots p_r}{p_b}\right)\right] + \phi\left(\frac{p_1p_2\cdots p_r}{p_b}\right)\\
 & \leq \frac{1}{p_b^{\epsilon}}\cdot \left[\frac{p_1p_2\cdots p_r}{p_b} - 2\phi\left(\frac{p_1p_2\cdots p_r}{p_b}\right)\right] + \phi\left(\frac{p_1p_2\cdots p_r}{p_b}\right),
\end{align*}
where $\epsilon=1$ or $2$ according as $n_b = 2$ or $n_b\geq 3$. Put $v:=\phi\left(\frac{p_1p_2\cdots p_r}{p_a}\right)$ and $w:= \frac{1}{p_b^{\epsilon}}\cdot \left[\frac{p_1p_2\cdots p_r}{p_b} - 2\phi\left(\frac{p_1p_2\cdots p_r}{p_b}\right)\right] + \phi\left(\frac{p_1p_2\cdots p_r}{p_b}\right)$. Then
$$\beta_a^{n_a} - \beta_b^{n_b}=\frac{n}{p_1p_2\cdots p_r}\cdot (u_a -u_b) >  \frac{n}{p_1p_2\cdots p_r}\cdot (v-w).$$
We need to show that $v-w \geq 0$. It can be calculated that
\begin{align*}
v-w & = \frac{1}{p_b^{\epsilon}}\cdot\left[\left[p_b^{\epsilon}(p_b -p_a)+2\phi(p_a)\right]\cdot \phi\left(\frac{p_1p_2\cdots p_r}{p_ap_b}\right)- \frac{p_1p_2\cdots p_r}{p_b}\right].
\end{align*}
By the given hypothesis, we have $p_b^{\epsilon}(p_b -p_a)+2\phi(p_a)\geq (r-1)p_a$. Therefore,
\begin{align*}
v-w & \geq \frac{1}{p_b^{\epsilon}}\cdot\left[(r-1)p_a \cdot\phi\left(\frac{p_1p_2\cdots p_r}{p_ap_b}\right)- \frac{p_1p_2\cdots p_r}{p_ap_b}\cdot p_a\right] \geq 0,
\end{align*}
where the last inequality follows using (\ref{eqn-2}).
\end{proof}

\begin{corollary}\label{cor-compa-3}
Let $a\in [r-1]$ such that $2\phi\left(\frac{p_1p_2\cdots p_{r}}{p_a}\right) < \frac{p_1p_2\cdots p_{r}}{p_a}$. Then the following hold:
\begin{enumerate}
\item[(i)] If $n_r\geq 3$, then $\beta_a^{n_a} > \beta_r^{n_r}$.
\item[(ii)] If $n_r=2$ and $p_r -p_{a}\geq r-3$, then $\beta_a^{n_a} > \beta_r^{n_r}$.
\end{enumerate}
\end{corollary}

\begin{proof}
Take $b=r$ in Proposition \ref{compa-3}. We have $p_{r}^2 (p_r -p_a)+2\phi(p_a) > p_r^2 > (r-1)p_a$. So (i) follows from Proposition \ref{compa-3}(i).

If $r=3$, then $p_r -p_{a}\geq 2$ and so $p_r(p_r -p_{a})+2\phi(p_{a}) > 2p_r > 2p_a =(r-1)p_a$. If $r\geq 4$ and $p_r -p_{a}\geq r-3$, then $p_r(p_r -p_{a})+2\phi(p_{a})\geq (p_a+2)(r-3)+2(p_a-1)\geq (r-1)p_a$. So (ii) follows from Proposition \ref{compa-3}(ii).
\end{proof}

\section{\bf The upper bounds $\theta_{a,b}^{s,t}$ for $\kappa(\mathcal{P}(C_n))$}

In this section, we consider $r\geq 3$. Let $a,b\in [r]$ with $a\neq b$, $s\in [n_a]$ and $t\in [n_b]$. We first determine the size of the cut-set $X_{a,b}^{s,t}$ of $\mathcal{P}(C_n)$.
Define the integer $\theta_{a,b}^{s,t}$ by
\begin{equation}\label{eqn-6}
	\theta_{a,b}^{s,t}:=\left\{\begin{array}{ll}
		\phi(n)+\frac{n}{p_1p_2\cdots p_r}\cdot v_1, & \text{if } s=n_a,\; t=n_b\\
		\phi(n)+\frac{n}{p_1p_2\cdots p_r}\cdot v_2, & \text{if } s=n_a,\; t<n_b\\
		\phi(n)+\frac{n}{p_1p_2\cdots p_r}\cdot v_3, & \text{if } s<n_a,\; t=n_b\\
		\phi(n)+\frac{n}{p_1p_2\cdots p_r}\cdot v_4, & \text{if } s<n_a,\; t<n_b,\\
	\end{array}\right.
\end{equation}
where
\begin{equation*}
	\begin{aligned}
		&v_1=\phi\left(\frac{p_1p_2\cdots p_r}{p_ap_b}\right)\left(1-\frac{2}{p_a^{n_a-1}p_b^{n_b-1}}\right)+\phi\left(\frac{p_1p_2\cdots p_r}{p_a}\right)\\
        &\hspace{5.5cm} +\phi\left(\frac{p_1p_2\cdots p_r}{p_b}\right)+\frac{p_1p_2\cdots p_r}{p_a^{n_a}p_b^{n_b}},\\
		&v_2=\phi\left(\frac{p_1p_2\cdots p_r}{p_ap_b}\right)\left(1-\frac{1}{p_b^{t}}\right)+\phi\left(\frac{p_1p_2\cdots p_r}{p_a}\right)\left(1-
		\frac{2}{p_a^{n_a-1}p_b^{t}}\right)\\
		&\hspace{4.5cm}+\phi\left(\frac{p_1p_2\cdots p_r}{p_b}\right)\left(1-\frac{1}{p_b^{t}}\right)+\frac{p_1p_2\cdots p_r}{p_a^{n_a}p_b^{t}},\\
		&v_3=\phi\left(\frac{p_1p_2\cdots p_r}{p_ap_b}\right)\left(1-\frac{1}{p_a^{s}}\right)+\phi\left(\frac{p_1p_2\cdots p_r}{p_a}\right)\left(1-\frac{1}{p_a^{s}}\right)\\
		&\hspace{4.5cm}+\phi\left(\frac{p_1p_2\cdots p_r}{p_b}\right)\left(1-\frac{2}{p_a^{s}p_b^{n_b-1}}\right)+\frac{p_1p_2\cdots p_r}{p_a^{s}p_b^{n_b}},\\
		&v_4=\phi\left(\frac{p_1p_2\cdots p_r}{p_ap_b}\right)\left(1-\frac{1}{p_a^{s}}\right)\left(1-\frac{1}{p_b^{t}}\right)+\phi\left(\frac{p_1p_2\cdots p_r}{p_a}\right)\left(1-\frac{1}{p_a^{s}}\right)\\
		&\hspace{3cm}+\phi\left(\frac{p_1p_2\cdots p_r}{p_b}\right)\left(1-\frac{1}{p_b^{t}}\right)-\frac{2\phi(p_1p_2\cdots p_r)}{p_a^{s}p_b^{t}}+\frac{p_1p_2\cdots p_r}{p_a^{s}p_b^{t}}.\\
	\end{aligned}
\end{equation*}
We prove the following:

\begin{proposition}\label{thm-X-a-b-s-t}
$\left|X_{a,b}^{s,t}\right|=\theta_{a,b}^{s,t}$ for $a,b\in [r]$ with $a\neq b$, $s\in [n_a]$ and $t\in [n_b]$.
\end{proposition}

\begin{proof}
Recall that $X_{a,b}^{s,t}=H_{a,b}^{s,t}\cup K_{a,b}^{s,t}$, where $K_{a,b}^{s,t}$ is the set of all the non-generators of the cyclic subgroup $S_{\frac{n}{p_a^s p_b^t}}$ of $C_n$ and $H_{a,b}^{s,t}$ is the union of the mutually disjoint sets $E_{\frac{n}{p_a^ip_b^j}}$ with $0\leq i\leq s$, $0\leq j\leq t$ and $(i,j)\neq (s,t)$. The set $E_n$ corresponds to $(i,j)=(0,0)$.
We have
$$\left|K_{a,b}^{s,t}\right|=\left| S_{\frac{n}{p_a^s p_b^t}}\right| - \phi\left(\frac{n}{p_a^{s}p_b^{t}}\right)=\frac{n}{p_a^{s}p_b^{t}}-\phi\left(\frac{n}{p_a^{s}p_b^{t}}\right).$$
The set $H_{a,b}^{s,t}$ can be written as
$$H_{a,b}^{s,t}=E_n\cup \left(\underset{i\in [s]}\cup E_{\frac{n}{p_a^i}}\right)\cup\left(\underset{j\in [t]}\cup E_{\frac{n}{p_b^j}}\right)\cup\left(\underset{k\in [s],\; l\in [t]}\cup E_{\frac{n}{p_a^kp_b^l}}\right)\setminus E_{\frac{n}{p_a^{s}p_b^{t}}}.$$
We have $|E_n|=\phi(n)$. Taking $k=1$ in Lemma \ref{lem-2.1}, we get
\begin{align*}
\left|\underset{i\in [s]}\cup E_{\frac{n}{p_a^i}}\right|=\underset{i=1}{\overset{s}{\sum}} \left\vert E_{\frac{n}{p_a^{i}}} \right\vert & = \left\{
\begin{array}{ll}
    \frac{n}{p_1p_2\cdots p_r}\cdot\phi\left(\frac{p_1p_2\cdots p_r}{p_a}\right), & \text{if } s=n_a\\
    \frac{n}{p_1p_2\cdots p_r}\cdot\phi\left(\frac{p_1p_2\cdots p_r}{p_a}\right)\left(1-\frac{1}{p_a^{s}}\right), & \text{if } s<n_a\\
\end{array}
\right.
\end{align*}
and
\begin{align*}
\left|\underset{j\in [t]}\cup E_{\frac{n}{p_b^j}}\right| =\underset{j=1}{\overset{t}{\sum}} \left\vert E_{\frac{n}{p_b^{j}}} \right\vert & = \left\{
\begin{array}{ll}
    \frac{n}{p_1p_2\cdots p_r}\cdot\phi\left(\frac{p_1p_2\cdots p_r}{p_b}\right), & \text{if } t=n_b\\
    \frac{n}{p_1p_2\cdots p_r}\cdot\phi\left(\frac{p_1p_2\cdots p_r}{p_b}\right)\left(1-\frac{1}{p_b^{t}}\right), & \text{if } t<n_b.
\end{array}
\right.
\end{align*}
By Lemma \ref{lem-2.3}, we have
\begin{align*}
\left|\underset{k\in [s],\; l\in [t]}\cup E_{\frac{n}{p_a^kp_b^l}}\right|
& =\overset{s}{\underset{k=1}\sum} \overset{t}{\underset{l=1}\sum}\left|E_{\frac{n}{p_a^kp_b^l}}\right|\\
& =\left\{\begin{array}{ll}
    \frac{n}{p_1p_2\cdots p_r}\cdot\phi\left(\frac{p_1p_2\cdots p_r}{p_ap_b}\right), & \text{if } s=n_a,\; t=n_b \\
    \frac{n}{p_1p_2\cdots p_r}\cdot\phi\left(\frac{p_1p_2\cdots p_r}{p_ap_b}\right)\left(1-\frac{1}{p_b^{t}}\right), & \text{if } s=n_a,\; t<n_b \\
    \frac{n}{p_1p_2\cdots p_r}\cdot\phi\left(\frac{p_1p_2\cdots p_r}{p_ap_b}\right)\left(1-\frac{1}{p_a^{s}}\right), & \text{if } s<n_a,\; t=n_b \\
    \frac{n}{p_1p_2\cdots p_r}\cdot\phi\left(\frac{p_1p_2\cdots p_r}{p_ap_b}\right)\left(1-\frac{1}{p_a^{s}}\right)\left(1-\frac{1}{p_b^{t}}\right), & \text{if } s<n_a,\; t<n_b.
\end{array}
\right.
\end{align*}
We also have
\begin{equation*}
\left|E_{\frac{n}{p_a^{s}p_b^{t}}}\right| = \phi\left(\frac{n}{p_a^{s}p_b^{t}}\right)=\left\{\begin{array}{ll}
    \frac{n}{p_1p_2\cdots p_r}\cdot\frac{1}{p_a^{s-1}p_b^{t-1}}\cdot\phi\left(\frac{p_1p_2\cdots p_r}{p_ap_b}\right), & \text{if } s=n_a,\; t=n_b \\
    \frac{n}{p_1p_2\cdots p_r}\cdot\frac{1}{p_a^{s-1}p_b^{t}}\cdot \phi\left(\frac{p_1p_2\cdots p_r}{p_a}\right), & \text{if } s=n_a,\; t<n_b\\
    \frac{n}{p_1p_2\cdots p_r}\cdot\frac{1}{p_a^{s}p_b^{t-1}}\cdot\phi\left(\frac{p_1p_2\cdots p_r}{p_b}\right), & \text{if } s<n_a,\; t=n_b \\
    \frac{n}{p_1p_2\cdots p_r}\cdot\frac{1}{p_a^{s}p_b^{t}}\cdot \phi(p_1p_2\cdots p_r), & \text{if } s<n_a,\; t<n_b.
   \end{array}\right.
\end{equation*}
With all the above information in hand, we can calculate the value of $\left|H_{a,b}^{s,t}\right|$ and find $\left|X_{a,b}^{s,t}\right|=\theta_{a,b}^{s,t}$ using the fact that the sets $H_{a,b}^{s,t}$ and $K_{a,b}^{s,t}$ are disjoint.
\end{proof}

\begin{corollary}\label{cor-X-a-b-s-t}
$\kappa(\mathcal{P}(C_n))\leq \theta_{a,b}^{s,t}=\left|X_{a,b}^{s,t}\right|$ for $a,b\in [r]$ with $a\neq b$, $s\in [n_a]$ and $t\in [n_b]$.
\end{corollary}

In the following proposition, for distinct $a,b\in [r]$ with $n_b\geq 2$, we compare the bounds $\beta_a^{n_a}$ and $\theta_{a,b}^{n_a,t}$ for varying $t\in [n_b -1]$.

\begin{proposition}\label{compa-X-a-b-s-t}
Let $a,b\in [r]$ with $a\neq b$ and $n_b\geq 2$. Put $\alpha:= \left(2+\frac{p_a^{n_a}-2}{p_b}\right)\phi\left(\frac{p_1p_2\cdots p_r}{p_ap_b}\right) -\frac{p_1p_2\cdots p_r}{p_ap_b}$. Then the following hold:
\begin{enumerate}
\item[(i)] If $\alpha=0$, then $\beta_a^{n_a} = \theta_{a,b}^{{n_a},1}=\theta_{a,b}^{{n_a},2}=\cdots =\theta_{a,b}^{{n_a},n_b-1}$.
\item[(ii)] If $\alpha <0$, then $\beta_a^{n_a} > \theta_{a,b}^{{n_a},1} > \theta_{a,b}^{{n_a},2} >\cdots >\theta_{a,b}^{{n_a},n_b-1}$.
\item[(iii)]  If $\alpha >0$, then $\beta_a^{n_a} < \theta_{a,b}^{{n_a},1} < \theta_{a,b}^{{n_a},2}<\cdots <\theta_{a,b}^{{n_a},n_b-1}$.
\end{enumerate}
\end{proposition}

\begin{proof}
If $n_b\geq 3$, then consider $t$ with $1\leq t<n_b-1$ (so that $t+1 < n_b$). Using (\ref{eqn-6}), we can calculate that $\theta_{a,b}^{n_a,t+1}-\theta_{a,b}^{n_a,t} =\frac{n}{p_1p_2\cdots p_r}\cdot\left(\frac{1}{p_b^t}-\frac{1}{p_b^{t+1}}\right)\cdot \lambda$ and $\theta_{a,b}^{n_a,1}-\beta_a^{n_a} =\frac{n}{p_1p_2\cdots p_r}\cdot\left(1-\frac{1}{p_b}\right)\cdot \lambda$,
where
\begin{align*}
\lambda & = \phi\left(\frac{p_1p_2\cdots p_r}{p_ap_b}\right)\left(1+\phi(p_a)+\frac{2\phi(p_b)}{p_a^{n_a-1}}\right)-\frac{p_1p_2\cdots p_r}{p_a^{n_a}}\\
&=\frac{p_b}{p_a^{n_a -1}}\cdot \left[\phi\left(\frac{p_1p_2\cdots p_r}{p_ap_b}\right)\left(2+\frac{p_a^{n_a}-2}{p_b}\right)-\frac{p_1p_2\cdots p_r}{p_ap_b}\right]=\frac{p_b}{p_a^{n_a -1}}\cdot\alpha.
\end{align*}
Note that $\lambda$ is independent of $t$. Further, $\lambda =0$ (respectively, $\lambda <0$, $\lambda >0$) if and only if $\alpha =0$ (respectively, $\alpha <0$, $\alpha >0$). Then (i), (ii) or (iii) holds according as $\alpha =0$, $\alpha <0$ or $\alpha >0$.
\end{proof}

%
%

\section{\bf Some useful results}

In this section, we consider $r\geq 4$. Let $X$ be a minimum cut-set of $\mathcal{P}(C_n)$. If $d$ is a positive divisor of $n$, then either $E_d$ is a subset of $X$ or $E_d$ is disjoint from $X$ by \cite[Lemma 2.3]{cps}. Define the subset $W(X)$ of $[r]$ by
$$W(X):= \left\{i\in [r]: E_{\frac{n}{p_i}}\subseteq X\right\}.$$
Then $|W(X)|\in\{0,1,2\}$ by \cite[Proposition 4.3]{MPS}.

\begin{proposition}\label{must-r}
Let $t\in [r]$ such that $p_{t}\geq r-1$. If $|W(X)|=2$ and $n_{t} \geq 2$, then $W(X)\cap \{t+1,t+2,\ldots,r\}\neq \emptyset$.
\end{proposition}

\begin{proof}
Let $W(X)=\{a,b\}$ for some $a,b\in [r]$ with $a < b$. We show that $b > t$. If possible, suppose that $b\leq t$. Then $a < t$. Using inequalities (\ref{eqn-0}) and (\ref{eqn-1}), we get
$$\left\vert E_{\frac{n}{p_a}} \right\vert = \phi\left(\frac{n}{p_a}\right)\geq p_{t}^{n_{t}-1} \phi\left(\frac{n}{p_{t}^{n_{t}}}\right)=\frac{n}{p_1p_2\cdots p_r}\cdot \phi\left(\frac{p_1p_2\cdots p_{r}}{p_{t}} \right)$$
and
$$\left\vert E_{\frac{n}{p_b}} \right\vert = \phi\left(\frac{n}{p_b}\right)\geq \phi\left(\frac{n}{p_{t}}\right)= \frac{n}{p_1p_2\cdots p_r}\cdot \frac{\phi\left(p_1p_2\cdots p_{r} \right)}{p_{t}}.$$
We know that $\beta_{t}^{n_{t}}=\phi(n)+ \frac{n}{p_1p_2\cdots p_r}\cdot u_t$, where
$$u_t= \phi\left(\frac{p_1p_2\cdots p_{r}}{p_{t}} \right) + \frac{1}{p_{t}^{n_{t}-1}}\cdot \left[\frac{p_1p_2\cdots p_{r}}{p_{t}} -2\phi\left(\frac{p_1p_2\cdots p_{r}}{p_{t}} \right) \right].$$
Put $\lambda = \phi\left(\frac{p_1p_2\cdots p_{r}}{p_{t}} \right) + \frac{\phi\left(p_1p_2\cdots p_{r} \right)}{p_{t}} - u_t$. It can be calculated that
\begin{align*}
\lambda & = \frac{1}{p_{t}^{n_{t}-1}}\left[\phi\left(\frac{p_1p_2\cdots p_{r}}{p_{t}} \right) \left[\phi\left(p_{t}^{n_{t} -1}\right)+2\right]- \frac{p_1p_2\cdots p_{r}}{p_{t}} \right].
\end{align*}
Since $n_{t}\geq 2$, we have $\phi\left(p_{t}^{n_{t} -1}\right)+2 \geq \phi\left(p_{t}\right)+2=p_t +1\geq r$. This gives
$$\lambda \geq \frac{1}{p_t^{n_{t} -1}}\left[r\cdot\phi\left(\frac{p_1p_2\cdots p_{r}}{p_{t}} \right) - \frac{p_1p_2\cdots p_{r}}{p_{t}} \right] > 0,$$
where the last inequality follows from (\ref{eqn-2}) as $r\geq 4$. Therefore, $\phi\left(\frac{p_1p_2\cdots p_{r}}{p_{t}} \right) + \frac{\phi\left(p_1p_2\cdots p_{r} \right)}{p_{t}} > u_t$. Since $E_1$, $E_n$, $E_{\frac{n}{p_a}}$, $E_{\frac{n}{p_b}}$ are mutually disjoint and contained in $X$, we get
\begin{align*}
|X| > \left\vert E_n\right\vert + \left\vert E_{\frac{n}{p_a}}\right\vert + \left\vert E_{\frac{n}{p_b}}\right\vert & = \phi(n)+ \frac{n}{p_1p_2\cdots p_r}\cdot\left[\phi\left(\frac{p_1p_2\cdots p_{r}}{p_{t}} \right) + \frac{\phi\left(p_1p_2\cdots p_{r} \right)}{p_{t}} \right]\\
 & > \phi(n)+ \frac{n}{p_1p_2\cdots p_r}\cdot u_t= \beta_{t}^{n_{t}}.
\end{align*}
This gives $\kappa(\mathcal{P}(C_n)) = |X|>\beta_{t}^{n_{t}}$, a contradiction to Proposition \ref{prop-Z-a-s}.
\end{proof}

The following corollary follows from Proposition \ref{must-r}. Note that the $k$-th prime is greater than or equal to $2k-1$.

\begin{corollary}\label{cor-must-r}
The following hold:
\begin{enumerate}
\item[(i)] If $n_{r} \geq 2$, then $|W(X)|\leq 1$.
\item[(ii)] If $n_{r-1} \geq 2$ and $|W(X)|=2$, then $r\in W(X)$.
\end{enumerate}
\end{corollary}

In the following proposition, we recall a few facts from \cite[Propositions 5.1, 5.2, Theorem 1.1]{MPS} concerning the possible candidates for minimum cut-sets of $\mathcal{P}(C_n)$ depending on $|W(X)|$.

\begin{proposition}[\cite{MPS}]\label{facts}
The following hold:
\begin{enumerate}
\item[(i)] If $W(X)=\emptyset$, then $X=Z_r^1$ and so $\kappa(\mathcal{P}(C_n))=|X|=\beta_r^1$.
\item[(ii)] If $W(X)=\{a\}$ for some $a\in [r]$, then $n_a\geq 2$, $2\phi\left(\frac{p_1p_2\cdots p_r}{p_a}\right) < \frac{p_1p_2\cdots p_r}{p_a}$, $X=Z_a^{n_a}$ and so $\kappa(\mathcal{P}(C_n))=|X|=\beta_a^{n_a}$.
\item[(iii)] If $W(X)=\{a,b\}$ for some $a,b\in [r]$ with $a\neq b$, then $X=X_{a,b}^{s,t}$ for some $s\in [n_a]$ and $t\in [n_b]$.
\end{enumerate}
\end{proposition}

In Proposition \ref{facts}(iii), it is necessary that $n_r = 1$  by Corollary \ref{cor-must-r}(i). The following proposition is useful for us.

\begin{proposition}\label{n-t=1}
Let $t\in [r]$ such that $n_t=1$ and $3\phi\left(\frac{p_1p_2\cdots p_r}{p_t}\right) \geq \frac{p_1p_2\cdots p_r}{p_t}$. Then $|W(X)\cap [t]|\leq 1$.
\end{proposition}

\begin{proof}
If possible, suppose that $W(X)=\{a,b\}$ for some $a,b\in [r]$ with $a < b\leq t$. Using inequality (\ref{eqn-1}), we have
$\left|E_{\frac{n}{p_a}}\right|=\phi\left(\frac{n}{p_a}\right)\geq\phi\left(\frac{n}{p_t}\right)$ and $\left|E_{\frac{n}{p_b}}\right|=\phi\left(\frac{n}{p_b}\right)\geq \phi\left(\frac{n}{p_t}\right)$. This gives
$\left|E_{\frac{n}{p_a}}\right|+ \left|E_{\frac{n}{p_b}}\right| \geq 2\phi\left(\frac{n}{p_t}\right)= \frac{n}{p_1p_2\cdots p_r}\cdot 2 \phi\left(\frac{p_1p_2\cdots p_r}{p_t}\right)$. As $n_t=1$, we know that
$$\beta_{t}^{1}= \phi(n)+ \frac{n}{p_1p_2\cdots p_r}\cdot \left[\frac{p_1p_2\cdots p_{r}}{p_{t}} -\phi\left(\frac{p_1p_2\cdots p_{r}}{p_{t}} \right) \right].$$
Since the mutually disjoint sets $E_1$, $E_n$, $E_{\frac{n}{p_a}}$ and $E_{\frac{n}{p_b}}$ are contained in $X$, we get
\begin{align*}
|X|-\beta_t^1 & > |E_n|+ \left|E_{\frac{n}{p_a}}\right| +\left|E_{\frac{n}{p_b}}\right|-\beta_t^1\\
 & \geq \frac{n}{p_1p_2\cdots p_r}\cdot \left[3\phi\left(\frac{p_1p_2\cdots p_r}{p_t}\right) - \frac{p_1p_2\cdots p_r}{p_t}\right] \geq 0.
\end{align*}
Then $\kappa(\mathcal{P}(C_n))=|X|>\beta_t^1$, a contradiction to Proposition \ref{prop-Z-a-s}.
\end{proof}

\noindent As a consequence of Proposition \ref{n-t=1}, we have the following:

\begin{corollary}\label{3_phi}
If $n_r=1$ and $3\phi(p_1p_2\cdots p_{r-1})\geq p_1p_2\cdots p_{r-1}$, then $|W(X)|\leq 1$.
\end{corollary}

\section{\bf Proof of Theorem \ref{main-1}}\label{sec-main-1}

In this section, we consider $r\geq 4$, $n_r\geq 2$ and $2\phi(p_1p_2\cdots p_{r-1}) < p_1p_2\cdots p_{r-1}$. Recall that
the subset $\Omega$ of $[r-1]$ is given by
$$\Omega=\left\{j\in [r-1]\mid n_j\geq 3,\; p_r-p_j \leq r-4,\; 2\phi\left(\frac{p_1p_2\cdots p_{r}}{p_j}\right) < \frac{p_1p_2\cdots p_{r}}{p_j}\right\}.$$

\begin{proof}[{\bf Proof of Theorem \ref{main-1}}]
Let $X$ be a minimum cut-set of $\mathcal{P}(C_n)$. As $n_r\geq 2$, we have $|W(X)|\in\{0,1\}$ by Corollary \ref{cor-must-r}(i). Since $2\phi(p_1p_2\cdots p_{r-1}) < p_1p_2\cdots p_{r-1}$, we have $\beta_r^1 > \beta_r^{n_r}\geq \kappa(\mathcal{P}(C_n))$ by Lemma \ref{comp-Z-a-s}(ii) and  Proposition \ref{prop-Z-a-s}. It then follows from Proposition \ref{facts}(i) that $|W(X)|\neq 0$.

Therefore, $|W(X)|=1$. Then, by Proposition \ref{facts}(ii), we have $X=Z_a^{n_a}$ for some $a\in [r]$ with $n_a\geq 2$ and $2\phi\left(\frac{p_1p_2\cdots p_r}{p_a}\right) < \frac{p_1p_2\cdots p_r}{p_a}$. Hence $\kappa(\mathcal{P}(C_n))=|X|=|Z_a^{n_a}|=\beta_a^{n_a}$.

Suppose that $a<r$. If $n_r\geq 3$, then $\kappa(\mathcal{P}(C_n))=\beta_a^{n_a}>\beta_r^{n_r}$ by Corollary \ref{cor-compa-3}(i), which is not possible by Proposition \ref{prop-Z-a-s}. Hence $n_r=2$. For completing the proof of Theorem \ref{main-1}, we need to prove that: $a\in\Omega$ and $a$ is the largest integer in $\Omega$.

We first show that $n_a\geq 3$ and $p_r-p_a \leq r-4$, so that $a$ would be in $\Omega$. If $n_a = 2$, then $p_a^{n_a -1}=p_a < p_r=p_r^{n_r -1}$ and so $\kappa(\mathcal{P}(C_n))=\beta_a^{n_a}>\beta_r^{n_r}$ by Lemma \ref{compa-2}. If $p_r -p_a\geq r-3$, then $\kappa(\mathcal{P}(C_n))=\beta_a^{n_a}>\beta_r^{n_r}$ by Corollary \ref{cor-compa-3}(ii). In both cases, we get a contradiction to Proposition \ref{prop-Z-a-s}. Hence $a\in\Omega$.

We now show that $a$ is the largest integer in $\Omega$. If possible, suppose that there exists $b\in\Omega$ with $a < b$. Then $n_b\geq 3$ and $p_r-p_b\leq r-4$. Since $p_r\geq 2r-1$, we get $p_b\geq p_r-(r-4)\geq r+3 > r-1$. So  $p_{b}^2 (p_b -p_a)+2\phi(p_a)>p_b^2 > (r-1)p_a$. Then, by Proposition \ref{compa-3}(i), we get $\kappa(\mathcal{P}(C_n))=\beta_a^{n_a} > \beta_b^{n_b}$, contradicting Proposition \ref{prop-Z-a-s}. Hence $a$ is the largest integer in $\Omega$. This completes the proof.
\end{proof}

In the following example, we provide integers $n$ with $r=6$, $n_r=2$ and $\Omega\neq \emptyset$ such that each of the possibilities for $\kappa(\mathcal{P}(C_n))$ in Theorem \ref{main-1}(ii) occurs.

\begin{example}\label{example}
Let $n=2^{n_1} \cdot 3^{n_2} \cdot 5^{n_3}  \cdot 7^{n_4} \cdot 11^{n_5} \cdot 13^{2}$ with $n_5\geq 3$. Here $r=6$, $p_1=2$, $p_2=3$, $p_3=5$, $p_4=7$, $p_5=11$, $p_6=13$ and $n_r=n_6=2$. As $p_1=2$, we have $2\phi\left(\frac{p_1p_2\cdots p_6}{p_5}\right) < \frac{p_1p_2\cdots p_6}{p_5}$. Since $p_r-p_5=2=r-4$ and $n_5\geq 3$, we get $\Omega = \{ 5 \}$. By Theorem \ref{main-1}(ii), we know that $\kappa(\mathcal{P}(C_n))\in\{\beta_5^{n_5},\beta_6^ 2\}$. Put $\omega:=p_1p_2\cdots p_6$. For $j\in\{5,6\}$, we have $\beta_j^{n_j}=\phi(n)+\frac{n}{\omega}\cdot u_j$, where
\begin{align*}
u_6 & =\frac{1}{13}\cdot \left[\frac{\omega}{13} + (13- 2)\phi\left(\frac{\omega}{13}\right)\right]=\frac{7590}{13},\\
u_5 &= \frac{1}{11^{n_5-1}}\cdot \left[\frac{\omega}{11}-2\phi\left(\frac{\omega}{11}\right)\right]+\phi\left(\frac{\omega}{11}\right)=\frac{1}{11^{n_5-1}}\cdot 1578 + 576.
\end{align*}
If $n_5=3$, then $u_5=\frac{1578}{121}+576 > \frac{7590}{13}=u_6$. If $n_5\geq 4$, then $u_5\leq \frac{1578}{1331}+576 < \frac{7590}{13}=u_6$. It then follows that $\beta_6^ 2\neq \beta_5^{n_5}$, and that $\kappa(\mathcal{P}(C_n))=\beta_6^ 2$ if $n_5=3$ and $\kappa(\mathcal{P}(C_n))=\beta_5^{n_5}$ if $n_5\geq 4$. This implies that if $X$ is a minimum cut-set of $\mathcal{P}(C_n)$, then $X=Z_6^2$ or $Z_5^{n_5}$ according as $n_5=3$ or $n_5\geq 4$.
\end{example}

\section{\bf Proof of Theorem \ref{main-2}}\label{sec-main-2}

In this section, we consider $r\in \{4,5\}$, $n_r=1$, $2\phi(p_1p_2\cdots p_{r-1}) < p_1p_2\cdots p_{r-1}$ and prove Theorem \ref{main-2}.
If $p_1\geq 5$, then the inequality (\ref{eqn-1-1}) gives that $2\phi(p_1p_2p_3) > p_1p_2p_3$ when $r=4$, and $2\phi(p_1p_2p_3p_4) > p_1p_2p_3p_4$ when $r=5$, which would contradict our assumption. Therefore, $p_1\in\{2,3\}$.

\subsection{The case $r=4$}

When $r=4$ and $n_r=1$, we first prove that $|W(X)|\leq 1$ for any minimum cut-set $X$ of $\mathcal{P}(C_n)$. We need the following lemma when $p_1=2$ and $p_2=3$.

\begin{lemma}\label{lem:r=4}
Let $n=p_1^{n_1}p_2^{n_2}p_3^{n_3}p_4$ with $p_1=2$, $p_2=3$, and $X$ be a minimum cut-set of $\mathcal{P}(C_n)$. If $|W(X)|=2$, then $4\in W(X)$.
\end{lemma}

\begin{proof}
Let $W(X)=\{k,l\}$ for some $k,l\in [4]$ with $k>l$. We claim that $k=4$. If $n_3\geq 2$, then the claim follows from Corollary \ref{cor-must-r}(ii) as $r=4$.

So assume that $n_3=1$. If possible, suppose that $l<k\leq 3$. Then, by inequality (\ref{eqn-1}), we have $|E_{\frac{n}{p_k}}|=\phi\left(\frac{n}{p_k}\right)\geq \phi\left(\frac{n}{p_3}\right)$ and $|E_{\frac{n}{p_l}}|=\phi\left(\frac{n}{p_l}\right)> \phi\left(\frac{n}{p_3}\right)$. Since $X$ contains the pair-wise disjoint sets $E_n$, $E_{\frac{n}{p_l}}$ and $E_{\frac{n}{p_k}}$, we get \begin{align*}
\kappa(\mathcal{P}(C_n))=|X| & \geq |E_n|+ |E_{\frac{n}{p_l}}|+|E_{\frac{n}{p_k}}|\\ & >\phi(n)+2\phi\left(\frac{n}{p_3}\right)= \phi(n)+ \frac{2n}{p_1p_2p_3p_4}\cdot \phi(p_1p_2p_4).
\end{align*}
Then, using (\ref{eqn-5}), we have
\begin{align*}
		\kappa(\mathcal{P}(C_n))-\beta_4^1 & >\frac{n}{p_1p_2p_3p_4}\cdot[2\phi(p_1p_2p_4)-p_1p_2p_3+\phi(p_1p_2p_3)]\\
		&=\frac{n}{p_1p_2p_3p_4}\cdot[4\phi(p_4)+2\phi(p_3)-6p_3]\\
		&\geq \frac{n}{p_1p_2p_3p_4}\cdot[4(p_3+1)+2(p_3-1)-6p_3]=\frac{2n}{p_1p_2p_3p_4}>0.
\end{align*}
This gives $\kappa(\mathcal{P}(C_n))> \beta_4^1$, a contradiction to Proposition \ref{prop-Z-a-s}. Therefore, $4\in W(X)$.
\end{proof}

\begin{proposition}\label{four_primes}
Let $r=4$, $n_4=1$ and $X$ be a minimum cut-set of $\mathcal{P}(C_n)$. Then  $|W(X)|\leq 1$.
\end{proposition}

\begin{proof}
If $p_1= 3$, then $3\phi(p_1p_2p_3) > p_1p_2p_3$ by Lemma \ref{p1_lower_bound} and so $|W(X)|\leq 1$ by Corollary \ref{3_phi}. If $p_1=2$ and $p_2\geq 5$, then $p_3\geq 7$ and so
\begin{align*}
\frac{\phi(p_1p_2p_3)}{p_1p_2p_3}=\left(1-\frac{1}{p_1}\right)\left(1-\frac{1}{p_2}\right)\left(1-\frac{1}{p_3}\right)
\geq\left(1-\frac{1}{2}\right)\left(1-\frac{1}{5}\right)\left(1-\frac{1}{7}\right)=\frac{12}{35}>\frac{1}{3}.
\end{align*}
This gives $3\phi(p_1p_2p_3)>p_1p_2p_3$ and hence $|W(X)|\leq 1$ by Corollary \ref{3_phi}. So consider that $p_1=2$ and $p_2=3$. Then $n=2^{n_1}3^{n_2}p_3^{n_3}p_4$. By Proposition \ref{facts}, we have $|W(X)|\leq 2$.

If possible, suppose that $|W(X)|=2$. Then $4\in W(X)$ by Lemma \ref{lem:r=4}. Let $W(X)=\{l,4\}$ for some $l\in[3]$. By Proposition \ref{facts}(iii), we have $X= X_{l,4}^{t,1}$ for some $t\in [n_l]$ and so $\kappa(\mathcal{P}(C_n))=|X|=|X_{l,4}^{t,1}|=\theta_{l,4}^{t,1}$.

We claim that $\kappa(\mathcal{P}(C_n))=\theta_{l,4}^{n_l,1}$ or $\kappa(\mathcal{P}(C_n))=\theta_{l,4}^{1,1}=\beta_4^1$. If $n_l=1$, then $t=n_l$ and so $\kappa(\mathcal{P}(C_n))=\theta_{l,4}^{n_l,1}$. Assume that $n_l\geq 2$. We have $\kappa(\mathcal{P}(C_n))\leq \beta_4^{1}$ by Proposition \ref{prop-Z-a-s}. Put $\alpha:= \left(2+\frac{p_4-2}{p_l}\right)\phi\left(\frac{p_1p_2p_3}{p_l}\right) -\frac{p_1p_2p_3}{p_l}$. Since $\frac{p_4-2}{p_l}\geq 1$, we have $\alpha\geq 0$ by inequality \eqref{eqn-2}. Then, applying Proposition \ref{compa-X-a-b-s-t} (in which take $a=4$, $b=l$ and $n_a=n_4=1$), the following must hold:
\begin{enumerate}
\item[(i)] if $\alpha >0$, then $\kappa(\mathcal{P}(C_n))=\theta_{l,4}^{n_l,1}$,
\item[(ii)] if $\alpha =0$, then $\kappa(\mathcal{P}(C_n))=\theta_{l,4}^{n_l,1}$ or $\kappa(\mathcal{P}(C_n))=\theta_{l,4}^{1,1}=\beta_4^1$.
\end{enumerate}
So the claim follows.\medskip

\noindent{\bf Case I.} $\kappa(\mathcal{P}(C_n))=\theta_{l,4}^{n_l,1}:$ Using (\ref{eqn-5}) and (\ref{eqn-6}), it can be calculated that
\begin{equation*}
	\begin{aligned}
\theta_{l,4}^{n_l,1}-\beta_4^1 &= \frac{n}{p_1p_2p_3p_4}\cdot\bigg[\left(1-\frac{2}{p_l^{n_l-1}}+\phi(p_4)+2\phi(p_l)\right)
\phi\left(\frac{p_1p_2p_3}{p_l}\right)
         -\left(p_l-\frac{1}{p_l^{n_l-1}}\right)\cdot\frac{p_1p_2p_3}{p_l}\bigg].\\
    \end{aligned}
\end{equation*}
Since $1-\frac{2}{p_l^{n_l-1}}+\phi(p_4)+2\phi(p_l) -3\left(p_l-\frac{1}{p_l^{n_l-1}}\right)=\frac{1}{p_l^{n_l-1}}+p_4-(p_l+2)>0$, we get
$$\theta_{l,4}^{n_l,1}-\beta_4^1 > \frac{n}{p_1p_2p_3p_4}\cdot \left(p_l-\frac{1}{p_l^{n_l-1}}\right)\left[3
\phi\left(\frac{p_1p_2p_3}{p_l}\right)-\frac{p_1p_2p_3}{p_l}\right]\geq 0,$$
where the last inequality follows using \eqref{eqn-2}. This gives $\kappa(\mathcal{P}(C_n))=\theta_{l,4}^{n_l,1}>\beta_4^1$, which contradicts Proposition \ref{prop-Z-a-s}.\medskip

\noindent{\bf Case II.} $\alpha =0$ and $\kappa(\mathcal{P}(C_n))=\theta_{l,4}^{1,1}=\beta_4^1:$
Since $p_1=2$ and $p_2=3$, we have $\alpha=0$ if and only if $l=3$ and $p_4=p_3+2$ by equality case of \eqref{eqn-2}. Now, for $l=3$ and $p_4=p_3+2$, we get using (\ref{eqn-5}) that
\begin{align*}
\beta_4^1- \beta_3^{n_3} & =  \frac{n}{p_1p_2p_3p_4}\cdot\left[p_1p_2p_3-\phi(p_1p_2p_3)-\phi(p_1p_2p_4)-\frac{1}{p_3^{n_3-1}}\cdot [p_1p_2p_4-2\phi(p_1p_2p_4)]\right]\\
 & = \frac{n}{p_1p_2p_3p_4}\cdot\left[2p_3-\frac{1}{p_3^{n_3 -1}}(2p_3+8)\right]>0,
\end{align*}
where the last strict inequality holds as $p_3\geq 5$ and $n_3\geq 2$. It then follows that $\kappa(\mathcal{P}(C_n))=\beta_4^1 > \beta_3^{n_3}$, contradicting Proposition \ref{prop-Z-a-s}. This completes the proof.
\end{proof}

We note that Proposition \ref{four_primes} need not be true when $r\geq 5$. In the following example, we provide an integer $n$  with $r=5$,  $n_r=1$ and a minimum cut-set $X$ of $\mathcal{P}(C_n)$ such that $|W(X)|=2$.

\begin{example}\label{example-0}
Let $n=2\cdot 3\cdot 5 \cdot 11 \cdot 13$ and $Y$ be a minimum cut-set of $\mathcal{P}(C_n)$. Here $r=5$, $p_1=2$, $p_2=3$, $p_3=5$, $p_4=11$, $p_5=13$ and $2\phi(p_1p_2p_3p_4)=160 < 330=p_1p_2p_3p_4$. We know that $Y\in \{Z_5^1, X_{4,5}^{1,1}\}$ (see Section \ref{known-results}). It can be calculated using (\ref{eqn-5}) and (\ref{eqn-6}) that
$$|X_{4,5}^{1,1}|=\theta_{4,5}^{1,1}=\phi(n)+198< \phi(n)+250=\beta_5^1=|Z_5^1|.$$
Therefore, $Y=X_{4,5}^{1,1}$ and $|W(Y)|=2$.
\end{example}

\begin{proof}[{\bf Proof of Theorem \ref{main-2}(i)}]

Given that $r=4$, $p_1\in\{2,3\}$, $n_4=1$ and $2\phi(p_1p_2p_3) < p_1p_2p_3$. Let $X$ be a minimum cut-set of $\mathcal{P}(C_n)$. Then $|W(X)|\leq 1$ by Proposition \ref{four_primes}. So $X=Z_4^1$ or $X\in \left\{Z_a^{n_a}\mid a\in \{2,3\}, n_a\geq 2,\; 2\phi\left(\frac{p_1p_2p_3p_4}{p_a}\right) < \frac{p_1p_2p_3p_4}{p_a}\right\}$ by Propositions \ref{facts} and \ref{compa-1}. This gives
$$\kappa\left(\mathcal{P}\left(C_n\right)\right)\in\left\{\beta_4^1, \beta_i^{n_i}\mid i\in \{2,3\}, n_i\geq 2,\; 2\phi\left(\frac{p_1p_2p_3p_4}{p_i}\right) < \frac{p_1p_2p_3p_4}{p_i}\right\}.$$
Suppose that $n_2\geq 2$ and $2\phi\left(p_1p_3p_4\right)< p_1p_3p_4$. If $p_1=3$, then   $$\frac{\phi\left(p_1p_3p_4\right)}{p_1p_3p_4}=\left(1-\frac{1}{p_1}\right)\left(1-\frac{1}{p_3}\right)\left(1-\frac{1}{p_4}\right)
\geq\left(1-\frac{1}{3}\right)\left(1-\frac{1}{7}\right)\left(1-\frac{1}{11}\right)=\frac{40}{77}>\frac{1}{2}$$
gives that $2\phi\left(p_1p_3p_4\right)> p_1p_3p_4$, which is against our assumption. So $p_1=2$.
Then, using (\ref{eqn-5}) and the fact that $2\phi\left(p_1p_3p_4\right) < p_1p_3p_4$, we get
\begin{align*}
\beta_2^{n_2}-\beta_4^{1} & > \frac{n}{p_1p_2p_3p_4}\cdot\left[\phi\left(p_1p_3p_4\right)+\phi(p_1p_2p_3)- p_1p_2p_3\right]\\
& = \frac{n}{2p_2p_3p_4}\cdot\left[\phi\left(p_3p_4\right)+\phi(p_2p_3)- 2p_2p_3\right]\\
& = \frac{n}{2p_2p_3p_4}\cdot\left[(p_3-1)(p_4-1)-(p_2+1)(p_3+1)+2\right] >0.
\end{align*}
This implies that $\kappa\left(\mathcal{P}\left(C_n\right)\right)\neq \beta_2^{n_2}$ and so $X\neq Z_2^{n_2}$. Therefore, $X=Z_4^1$ or $Z_3^{n_3}$, where the latter possibility occurs only when $n_3\geq 2$ and $2\phi\left(p_1p_2p_4\right) < p_1p_2p_4$. This completes the proof.
\end{proof}

In the following example, we provide integers $n$ with $r=4$, $n_4=1$, $n_3\geq 2$ and $2\phi(p_1p_2p_3) < p_1p_2p_3$ such that $\kappa(\mathcal{P}(C_n))=\beta_3^{n_3}$ in Theorem \ref{main-2}(i).

\begin{example}\label{example-1}
Let $n=2^{n_1}\cdot 3^{n_2} \cdot 5^{n_3} \cdot 7$, where $n_3\geq 2$. Here $r=4$, $p_1=2$, $p_2=3$, $p_3=5$, $p_4=7$, $n_r=n_4=1$ and $2\phi(p_1p_2p_3)=16 < 30=p_1p_2p_3$.  By Theorem \ref{main-2}(i), we know that $\kappa\left(\mathcal{P}\left(C_n\right)\right)=\min \{\beta_3^{n_3},\beta_4^1\}$.
For $j\in\{3,4\}$, we have $\beta_j^{n_j}=\phi(n)+\frac{n}{p_1p_2p_3p_4}\cdot u_j$, where
\begin{align*}
u_4 & =p_1p_2p_3-\phi(p_1p_2p_3)=30- 8 =22,\\
u_3 &= \frac{1}{p_3^{n_3-1}}\cdot \left[p_1p_2p_4 - 2\phi\left(p_1p_2p_4\right)\right]+\phi\left(p_1p_2p_4\right)=\frac{18}{5^{n_3-1}} + 12.
\end{align*}
Since $n_3\geq 2$, we have $u_3\leq \frac{18}{5}+12 < 22=u_4$. It then follows that $\beta_4^ 1\neq \beta_3^{n_3}$ and $\kappa(\mathcal{P}(C_n))=\beta_3^{n_3}$. This implies that $Z_3^{n_3}$ is the only minimum cut-set of $\mathcal{P}(C_n)$.
\end{example}

\subsection{The case $r=5$}

We need the following lemma when $r=5$ and $p_1=3$.

\begin{lemma}\label{r=5}
Let $r=5$ and $p_1=3$. Then  $\phi\left(\frac{p_1p_2p_3p_4p_5}{p_i}\right)+\phi(p_1p_2p_3p_4)> p_1p_2p_3p_4$ for $i\in \{1,2,3\}$.
\end{lemma}

\begin{proof}
Since $\phi\left(p_2p_3p_4p_5\right)> \phi\left(p_1p_3p_4p_5\right)> \phi\left(p_1p_2p_4p_5\right)$, it is enough to prove that
$\phi\left(p_1p_2p_4p_5\right)+\phi(p_1p_2p_3p_4)> p_1p_2p_3p_4$. As $p_1=3$, we need to show that  $$2\left[\phi(p_5)+\phi(p_3)\right]\phi(p_2p_4)>3p_2p_3p_4.$$
Since $p_3$ is a prime and $p_3 >3$, one of $p_3+2$ and $p_3+4$ must be divisible by $3$. So $p_3,p_4,p_5$ cannot be consecutive twin primes and hence $p_5-p_3\geq 6$. Then
$$2\left[\phi(p_5)+\phi(p_3)\right]\geq 4(p_3+2)>3\cdot \frac{5}{4}\cdot (p_3+2)> 3\cdot \frac{5}{4}\cdot \frac{p_3p_4}{\phi(p_4)}.$$
As $p_2\geq 5$, we have $\frac{5}{4}\geq \frac{p_2}{\phi(p_2)}$. Then
$2\left[\phi(p_5)+\phi(p_3)\right]> 3\cdot \frac{p_2}{\phi(p_2)}\cdot \frac{p_3p_4}{\phi(p_4)}$ and hence
$2\left[\phi(p_5)+\phi(p_3)\right]\phi(p_2p_4)>3p_2p_3p_4.$
\end{proof}

\begin{proof}[{\bf Proof of Theorem \ref{main-2}(ii)}]
Given that $r=5$, $p_1 = 3$, $n_r=1$ and $2\phi(p_1p_2p_3p_{4}) < p_1p_2p_3p_{4}$.
Let $X$ be a minimum cut-set of $\mathcal{P}(C_n)$.
Since $p_1= 3$, we have $3\phi(p_1p_2p_3p_{4}) > p_1p_2p_3p_{4}$ by Lemma \ref{p1_lower_bound}, and so $|W(X)|\leq 1$ by Corollary \ref{3_phi}. Then, by Propositions \ref{facts} and \ref{compa-1}, we get $X=Z_5^1$ or $X\in \left\{Z_a^{n_a}\mid a\in \{2,3,4\}, n_a\geq 2,\; 2\phi\left(\frac{p_1p_2p_3p_4p_5}{p_a}\right) < \frac{p_1p_2p_3p_4p_5}{p_a}\right\}$ and hence
$$\kappa\left(\mathcal{P}\left(C_n\right)\right)\in\left\{\beta_5^1, \beta_i^{n_i}\mid i\in \{2,3,4\}, n_i\geq 2,\; 2\phi\left(\frac{p_1p_2p_3p_4p_5}{p_i}\right) < \frac{p_1p_2p_3p_4p_5}{p_i}\right\}.$$
Let $i\in \{2,3\}$ such that $n_i\geq 2$ and $2\phi\left(\frac{p_1p_2p_3p_4p_5}{p_i}\right)< \frac{p_1p_2p_3p_4p_5}{p_i}$. Then, using (\ref{eqn-5}), it can be seen that
$$\beta_i^{n_i}-\beta_5^{1} > \frac{n}{p_1p_2\cdots p_5}\cdot\left[\phi\left(\frac{p_1p_2p_3p_4p_5}{p_i}\right)+\phi(p_1p_2p_3p_4)- p_1p_2p_3p_4\right]>0,$$
where the last inequality holds by Lemma \ref{r=5}. This implies that $\kappa\left(\mathcal{P}\left(C_n\right)\right)\notin \{\beta_2^{n_2},\beta_3^{n_3}\}$ by Proposition \ref{prop-Z-a-s} and so $X\notin\{Z_2^{n_2}, Z_3^{n_3}\}$. Thus, $X=Z_5^1$ or $Z_4^{n_4}$, where the latter possibility occurs only when $n_4\geq 2$ and $2\phi\left(p_1p_2p_3p_5\right) < p_1p_2p_3p_5$. This completes the proof.
\end{proof}

In the following example, we provide integers $n$ with $r=5$, $p_1=3$, $n_5=1$, $n_4\geq 2$ and $2\phi(p_1p_2p_3p_4) < p_1p_2p_3p_4$ such that $\kappa(\mathcal{P}(C_n))=\beta_4^{n_4}$ in Theorem \ref{main-2}(ii).

\begin{example}\label{example-2}
Let $n=3^{n_1} \cdot 5^{n_2} \cdot 7^{n_3} \cdot 11^{n_4} \cdot 13$, where $n_4\geq 2$. Here $r=5$, $p_1=3$, $p_2=5$, $p_3=7$, $p_4=11$, $p_5=13$, $n_r=n_5=1$ and $2\phi(p_1p_2p_3p_4)= 960 < 1155=p_1p_2p_3p_4$. By Theorem \ref{main-2}(ii), we know that $\kappa\left(\mathcal{P}\left(C_n\right)\right)=\min\{\beta_4^{n_4},\beta_5^1\}$. For $j\in\{4,5\}$, we have $\beta_j^{n_j}=\phi(n)+\frac{n}{p_1p_2p_3p_4p_5}\cdot w_j$, where
\begin{align*}
w_5 & =p_1p_2p_3p_4 -\phi(p_1p_2p_3p_4) =675,\\
w_4 &=\frac{1}{11^{n_4-1}}[p_1p_2p_3p_5 - 2\phi(p_1p_2p_3p_5)]+\phi(p_1p_2p_3p_5) = \frac{1}{11^{n_4 -1}}\cdot 213 +576.
\end{align*}
Since $n_4\geq 2$, we have $w_4\leq \frac{213}{11}+ 576 < 675=w_5$. It then follows that $\beta_5^ 1\neq \beta_4^{n_4}$ and that $\kappa(\mathcal{P}(C_n))=\beta_4^{n_4}$. This implies that $Z_4^{n_4}$ is the only minimum cut-set of $\mathcal{P}(C_n)$.
\end{example}

\vskip .5cm

\noindent{\bf Acknowledgements.} Sanjay Mukherjee is supported by Grant No. 09/1248(0004)/2019-EMR-I of the Council of Scientific and Industrial Research (CSIR), Government of India. Kamal Lochan Patra is partially supported by Project No. MTR/2022/000424 of the Science and Engineering Research Board, Government of India. Binod Kumar Sahoo is partially supported by Project No. CRG/2022/000344 of the Science and Engineering Research Board, Government of India.

\vskip .5cm

\noindent{\bf Addresses}: Sanjay Mukherjee, Kamal Lochan Patra, Binod Kumar Sahoo

\begin{enumerate}
\item[1)] School of Mathematical Sciences\\
National Institute of Science Education and Research (NISER), Bhubaneswar\\
P.O.- Jatni, District- Khurda, Odisha--752050, India

\item[2)] Homi Bhabha National Institute (HBNI)\\
Training School Complex, Anushakti Nagar, Mumbai--400094, India
\end{enumerate}

\noindent {\bf E-mails}: sanjay.mukherjee@niser.ac.in, klpatra@niser.ac.in, bksahoo@niser.ac.in
\end{document}